\documentclass{amsart}
\usepackage{amssymb,latexsym,graphics,enumerate,color}
\usepackage{fullpage}
\usepackage{color}

\usepackage{bbm}

\numberwithin{equation}{section}

\usepackage[normalem]{ulem}

\usepackage{xcolor,cancel}

\newtheorem{theorem}{Theorem}[section]
\newtheorem{lemma}[theorem]{Lemma}
\newtheorem{proposition}[theorem]{Proposition}
\newtheorem{corr}[theorem]{Corollary}

\newcommand{\bth}{\begin{theorem}}
\newcommand{\ble}{\begin{lemma}}
\newcommand{\bcor}{\begin{corr}}

 \newcommand{\bdeff}{\begin{deff}}

\newcommand{\bprop}{\begin{proposition}}
\newcommand{\ele}{\end{lemma}}
\newcommand{\ecor}{\end{corr}}
\newcommand{\edeff}{\end{deff}}

\newcommand{\eprop}{\end{proposition}}

\renewcommand{\Pi}{\varPi}

\renewcommand{\epsilon}{\varepsilon}

\newcommand{\parital}{\partial}

\newcommand{\R}{{\mathbb R}}

\newcommand{\N}{{\mathbb N}}

\newcommand{\la}{{\langle}}
\newcommand{\ra}{{\rangle}}
\newcommand{\cd}{{\,\cdot\,}}

\newcommand{\ang}{{\not\negmedspace\nabla}}
\newcommand{\good}{{\not\negmedspace\partial}}

\renewcommand{\S}{{\mathbb{S}}}
\newcommand{\p}{{\tilde{p}}}

\newcommand{\vertiii}[1]{{\left\vert\kern-0.25ex\left\vert\kern-0.25ex\left\vert #1 
    \right\vert\kern-0.25ex\right\vert\kern-0.25ex\right\vert}}

\begin{document}
\bibliographystyle{plain}

\title
{Global existence for quasilinear wave equations satisfying the null
  condition}

\author{Michael Facci}

\author{Jason Metcalfe}
\email{mfacci@live.unc.edu}
\email{metcalfe@email.unc.edu}
\address{Department of Mathematics, University of North Carolina, Chapel Hill, NC 27599-3250}

\thanks{The second author was supported in part by Simons Foundation
  Collaboration Grant 711724 and NSF grant DMS-2054910.}


\begin{abstract} We explore the global existence of solutions to
  systems of quasilinear wave equations satisfying the null condition
  when the initial data are sufficiently small.  We adapt an approach
  of Keel, Smith, and Sogge, which relies on integrated local energy
  estimates and a weighted Sobolev estimate that yields decay in $|x|$, by
  using the $r^p$-weighted local energy estimates of Dafermos and
  Rodnianski.  One advantage of this approach is that all
  time-dependent vector fields can be avoided and the proof can be
  readily adapted to address wave equations exterior to star-shaped obstacles.
\end{abstract}

\maketitle

\section{Introduction}
This article focuses on re-examining the proof of small data global existence for
systems of wave equations satisfying the classical null condition in
$(1+3)$-dimensions.  The proof relies only on the translational and
rotational symmetries of the d'Alembertian.  No explicit decay in time
is required.  Instead, in the spirit of the almost global existence
proofs of \cite{KSSsem} and \cite{MS_SIAM}, a weighted
Sobolev estimate that provides decay in $|x|$ is paired with a local
energy estimate.  In this case, however, for the ``good'' derivatives
that the null condition promises, we use the
$r^p$-weighted local energy estimate of \cite{DafRod}.  When
considering quasilinear equations one in essence has geometry that
depends on the solution while the solution in turn depends on the
geometry.  The highest order estimates need to be adapted to this
geometry.  Upon performing the typical manipulations for the
$r^p$-weighted estimates, the possibility that the ``good'' derivative
from the multiplier and the ``good'' derivative from the null
condition land on the same factor is encountered.
Our method introduces
a simple approach to this issue, which \cite{keir} calls
``the problem of multiple good derivatives'', by allowing for different choices of
$p$ in the $r^p$-weighted estimate.

Specifically, we will consider
\begin{equation}
  \label{main_equation}
\begin{split}  \Box u^I& = A^{I,\alpha\beta}_{JK} \partial_\alpha u^J \partial_\beta
  u^K + B^{I,\gamma\alpha\beta}_{JK} \partial_\gamma u^J
  \partial_\alpha\partial_\beta u^K,\\
  u&(0,\cd)=f,\quad \partial_tu(0,\cd)=g.
\end{split}
\end{equation}
Here $(t,x)\in \R_+\times\R^3$, $I=1,2,\dots, M$, and $u=(u^1,\dots,u^M)$.  Repeated Greek
indices are implicitly summed from $0$ to $3$ where
$\partial_0=\partial_t$, lower case Latin indices are summed from $1$
to $3$, and repeated upper case indices are summed
from $1$ to $M$.  The coefficients of the quasilinear terms are
assumed to satisfy
the symmetries
\begin{equation}
  \label{Bsym}
  B^{I,\gamma\alpha\beta}_{JK} = B^{I,\gamma\beta\alpha}_{JK} = B^{K,\gamma\alpha\beta}_{JI}.
\end{equation}
We have truncated \eqref{main_equation} at the quadratic level.  As is
well-known, for problems with small initial data, higher order terms
are typically better behaved.

Even for small, sufficiently regular and decaying initial data,
solutions to \eqref{main_equation} can only be ensured to exist almost
globally, which means that the lifespan of the solution grows
exponentially as the size of the initial data shrinks.  In \cite{christodoulou_null}
and \cite{klainerman_null}, the {\em null condition} was identified as a
sufficient condition for guaranteeing global solutions to
\eqref{main_equation} for small initial data.  We assume
the same here, which requires that
\begin{equation}
  \label{null_condition}
  A^{I,\alpha\beta}_{JK}\xi_\alpha\xi_\beta = 0,\qquad \text{and} \qquad
  B^{I,\gamma\alpha\beta}_{JK} \xi_\alpha\xi_\beta\xi_\gamma=0,\quad
  \text{whenever } \xi_0^2-\xi_1^2-\xi^2_2-\xi^2_3=0.
\end{equation}
These conditions promise that at least one factor of each nonlinear
term is a ``good'' derivative, which are directional derivatives in
directions that are tangent to the light cone $t=|x|$ and are known to
have more rapid decay.  We will fix the notation and more explicitly
describe these in the
next subsection.

The main result of this paper establishes global existence for
\eqref{main_equation} subject to \eqref{null_condition} for
sufficiently small initial data.
\begin{theorem}
 \label{thm:main}  Fix $0<\p<2$, and suppose
 $f, g\in C^\infty(\R^3)$ satisfy
 \begin{equation}
   \label{smalldata}
  \sum_{|\mu|\le N}\Bigl(\|\la r\ra^{\frac{\p}{2}+|\mu|} \partial^\mu \nabla
  f\|_{L^2} + \|\la r\ra^{\frac{\p}{2}+|\mu|} \partial^\mu g\|_{L^2} + \|\la
  r\ra^{\frac{\p-2}{2}+|\mu|} \partial^\mu f\|_{L^2}\Bigr)\le \varepsilon 
 \end{equation}
for $\varepsilon>0$ sufficiently small and $N$ sufficiently large.  Then provided that
\eqref{Bsym} and \eqref{null_condition} hold, \eqref{main_equation}
admits a global solution $u\in C^\infty(\R_+\times \R^3)$.
\end{theorem}

Here we have used $\la r\ra = \sqrt{1+r^2}$.  We note that the assumption \eqref{smalldata} could be refined after we
introduce some additional notation.

A common approach in proofs of long-time existence for nonlinear wave
equations is to rely on the method of invariant 
vector fields and the Klainerman-Sobolev inequality \cite{Klain}. 
Motivated by a desire to study similar wave equations in
exterior domains, with multiple speeds, and with nontrivial background
geometry, multiple approaches now exist that do not necessitate the
use of the full set of invariant vector fields.  A sample of such
results includes \cite{SdTu}, \cite{KSSquasobst}, \cite{MS_null_ext,
  MS_hyp_trap}, \cite{MNS_null, MNS_multi_speed}, \cite{lindblad_nakamura_sogge}.

The majority of the above results still rely on the scaling vector
field $t\partial_t+r\partial_r$.  Works such as
\cite{katayama_kubo}, \cite{keir}, \cite{Yang_quas}
 have pioneered methods for quasilinear equations that do not
rely on any time dependent vector field.  These methods rely on other
means of obtaining $t$-decay.  In the case of, e.g., \cite{keir} this
is accomplished by considering a null foliation.

In the current paper, we explore a technique that is more akin to
\cite{KSSsem}, \cite{MS_SIAM}.  Rather than relying on decay in $t$, this
approach couples decay in $|x|$ with local energy estimates for the
wave equation in order to obtain almost global existence without
assuming special structures on the nonlinearity.  In order to take
advantage of the good derivatives that the null condition ensures, we
will employ the $r^p$-weighted local energy estimate of
\cite{DafRod}.  We note that our method can immediately be adapted to
prove the same result for Dirichlet-wave equations exterior to
star-shaped obstacles.

In \cite{FMM}, this same approach was explored for semilinear wave
equations.  The current result is a bit more involved.  In order to
avoid a loss of regularity at the highest order, the estimates need to
be adapted to allow for small, time-dependent perturbations of
$\Box$.  Upon doing so, ``the problem of multiple good derivatives''
as described in \cite[Section 1.5.4]{keir} is encountered.  In a
typical term encountered within the $r^p$-weighted local energy
estimate, there is a cubic interaction.  Two factors arise from the
quadratic nonlinearity and one from the multiplier.  Amongst these
factors, the multiplier contributes a good derivative and the null
condition promises at least one additional good derivative.  If these
good derivatives fall on different factors, the method of \cite{FMM}
applies easily.  When
adapting the estimates to allow for the perturbations, the
manipulations allow for these good derivatives to both fall on the
same factor.  We propose an alternative to the methods of \cite{keir}
for subverting this problem.  In particular, we consider separately a
lower order energy and a high order energy, which on its surface is
commonplace.  When doing so, however, we consider different choices of
$p$ in the $r^p$-weighted estimates, and this will allow us to avoid
the use of time dependent vector fields to obtain the additional decay
needed for these terms.

\subsection{Notation}  Here fix the notation that will be used
throughout the paper.  We let $u'=\partial u = (\partial_t u,\nabla_x
u)$ denote the space-time gradient.  The notation
\[\Omega = (x_2\partial_3-x_3\partial_2, x_3\partial_1-x_1\partial_3,
  x_1\partial_2-x_2\partial_1)\]
is used for the generators of rotations.  And
\[Z=(\partial_0,\partial_1,\partial_2,\partial_3,\Omega_1,\Omega_2,\Omega_3)\]
will denote our collections of admissible vector fields.  We will use
the shorthand
\[|Z^{\le N} u|=\sum_{|\mu|\le N} |Z^{\mu} u|,\quad |\partial^{\le
    N} u| = \sum_{|\mu|\le N} |\partial^\mu u|.\]

The (spatial) gradient will be frequently (orthogonally) decomposed
into its radial and angular parts:
\[\nabla_x = \frac{x}{r}\partial_r + \ang.\]
Here, as is standard, $r=|x|$ and $\partial_r = \frac{x}{r}\cdot\nabla_x$.
The components of $\partial u$ that are tangent to the light cone are
known to have better decay properties.  We will abbreviate these
``good'' derivatives as
\[\good = (\partial_t+\partial_r, \ang).\]

A key property of the admissible vector fields is that they satisfy:
\[[\Box, Z]=0.\]
We also need to understand how they interact with $\partial$ and
$\good$.  In particular, we have
\begin{equation}
  \label{commutator}
|[Z, \partial]u|\le |\partial u|,\quad |[Z,\good] u|\le
\frac{1}{r}|Zu|, \quad |[\partial, \good]u|\le \frac{1}{r}|\partial u|.  
\end{equation}
In the second computation, we use the fact that $|\ang u|\le
\frac{1}{r}|Zu|$, which follows from $\ang = -\frac{x}{r^2}\times\Omega$.

\section{Local Energy Estimates}
In order to handle the quasilinear nature of the problem, we will rely
on linear estimates for the wave equation on geometries that are a
small, though time-dependent, perturbation of Minkowski space.  In
particular, we consider solutions to
\begin{equation}
  \label{pert_equation}
  \begin{split}
  (\Box_h u)^I&=F^I,\\
  u(0,\cd)=f,&\quad \partial_tu(0,\cd)=g
\end{split}
\end{equation}
where
 \[(\Box_h u)^I = (\partial_t^2 -\Delta)u^I + h^{I,\alpha\beta}_K
  \partial_\alpha\partial_\beta u^K\]
and the perturbations are assumed to satisfy
\begin{equation}
  \label{h_symmetry}
  h^{I,\alpha\beta}_K = h^{K,\alpha\beta}_I = h^{I,\beta\alpha}_K.
\end{equation}
For a differential operator $D$, we use the following
notation
\[|Dh|=\sum_{I,K=1}^M \sum_{\alpha,\beta=0}^3 |Dh^{I,\alpha\beta}_K(t,x)|,
\quad |D(h^{\alpha\beta}\omega_\alpha\omega_\beta)| = \sum_{I,K=1}^M |D(h^{I,\alpha\beta}_K(t,x)\omega_\alpha\omega_\beta)|.\]

In \cite{MS_null_ext}, the integrated local energy estimate was
established for $\Box_h$ and used to prove global existence for
systems of wave equations satisfying the null condition in exterior
domains. 
We will utilize the notation
\[\|u\|_{LE}=\sup_{R\ge 1} R^{-1/2}
  \|u\|_{L^2_tL^2_x([0,\infty)\times\{\frac{R}{2}\le \la x\ra \le R\})},\quad
  \|u\|_{LE^1} = \|(\partial u, u/r)\|_{LE}\]
 and record the following immediate corollary of \cite[Proposition 2.2]{MS_null_ext}.
\begin{proposition}
  Suppose that $h$ satisfies \eqref{h_symmetry} and
\begin{equation}
  \label{h_small}
  |h| = \sum_{I,K=1}^M\sum_{\alpha,\beta=0}^3
  |h^{I,\alpha\beta}_K(t,x)|\le \delta \ll 1.
\end{equation}
with $\delta>0$ sufficiently small.  Then if $u\in C^\infty$ solves
  \eqref{pert_equation} and for every $t$, $|\partial^{\le 1}
  u(t,x)|\to 0$ as $|x|\to \infty$, then
  \begin{multline}\label{LE}
    \|u\|^2_{LE^1} + \|\partial u\|^2_{L^\infty_t L^2_x}\lesssim
    \|\partial u(0,\cd)\|^2_{L^2} + \int_0^\infty\int
    \Bigl(|\partial u| + \frac{|u|}{\la r\ra}\Bigr) |\Box_h u|\,dx\,dt
   \\ + \int_0^\infty\int \Bigl|\partial_\alpha h^{I,\alpha\beta}_K
    \partial_\beta u^K\Bigr|\Bigl(|\partial u^I| +
    \frac{|u^I|}{r}\Bigr)\,dx\,dt
    +\int_0^\infty\int \Bigl|(\partial h^{I,\alpha\beta}_K)
    \partial_\beta u^K\partial_\alpha u^I\Bigr|\,dx\,dt
   \\ + \int_0^\infty\int \frac{|h|}{\la r\ra}|\partial u| \Bigl(|\partial u| +
    \frac{|u|}{r}\Bigr)\,dx\,dt.
  \end{multline}
\end{proposition}

The spatial portion of the $LE^1$ norm considers the local energy of
$u$ in an inhomogeneous annulus with a weight that is dictated by the
radii of the annulus.  The estimate captures the fact that this local
energy decays at a sufficiently rapid rate to permit
$L^2$-integrability in time with a bound that (essentially) matches
that provided by the energy estimate (for perturbations of $\Box$).

The proof of this proposition follows upon pairing $(\Box_h u)^I$ with
a multiplier of the form
\[C\partial_t u^I + \frac{r}{r+R}\partial_ru^I
+ \frac{1}{r+R} u^I,\] integrating over $[0,T]\times\R^3$, and integrating by parts.  See also \cite{Sterb}
and \cite{MS_SIAM}.

If we set $\omega=(-1,x/r)$, in order to take advantage of the null
condition in the sequel, we note
\begin{equation}
  \label{nullify}
 \partial_\alpha h^{I,\alpha\beta}_K \partial_\beta u^K
  =
  (\partial_\alpha-\omega_\alpha\partial_r)h^{I,\alpha\beta}_K \partial_\beta
  u^K + \omega_\alpha \partial_r h^{I,\alpha\beta}_K (\partial_\beta
  -\omega_\beta\partial_r)u^K + \partial_r(\omega_\alpha\omega_\beta
  h^{I,\alpha\beta}_K)\partial_r u^K,
\end{equation}
and
\begin{equation}
  \label{nullify2}
  \partial h^{I,\alpha\beta}_K \partial_\beta u^K \partial_\alpha
  u^I
= \partial h^{I,\alpha\beta}_K (\partial_\beta
-\omega_\beta \partial_r)u^K \partial_\alpha u^I
+ \partial h^{I,\alpha\beta} \omega_\beta \partial_r u^K
(\partial_\alpha - \omega_\alpha \partial_r) u^I
+ (\omega_\alpha\omega_\beta \partial h^{I,\alpha\beta}) \partial_r
u^K \partial_r u^I.
\end{equation}
Using these,
we observe that \eqref{LE} implies that
\begin{multline}
  \label{LEnull}
   \|u\|^2_{LE^1} + \|\partial u\|^2_{L^\infty_t L^2_x}\lesssim
    \|\partial u(0,\cd)\|^2_{L^2} + \int_0^\infty\int
    \Bigl(|\partial u| + \frac{|u|}{\la r\ra}\Bigr) |\Box_h u|\,dx\,dt
   \\+\int_0^\infty\int |\partial h| |\good u|\Bigl(|\partial
   u|+\frac{|u|}{r}\Bigr)\,dx\,dt
+\int_0^\infty \int \Bigl(\frac{|h|}{\la r\ra} + |\good h| +
|\partial(\omega_\alpha\omega_\beta h^{\alpha\beta})|\Bigr) |\partial
u|\Bigl(|\partial u|+\frac{|u|}{r}\Bigr)\,dx\,dt.
\end{multline}

We next consider a variant of the $r^p$-weighted local energy
estimate of \cite{DafRod}.  To, e.g., readily control commutators involving
vector fields and $\good$, the following estimate that is akin to a
Hardy inequality is convenient.
\begin{lemma}
Suppose that $u\in C^1([0,\infty)\times \R^3)$ and that for each $t\in
[0,\infty)$, $r^{p/2}|u(t,x)|\to 0$ as $|x|\to \infty$.  Then,
provided $0<p<2$,
  \begin{equation}
    \label{Hardy}
 \|r^{\frac{p-2}{2}}u\|_{L^\infty_t L^2_x} +  \|r^{\frac{p-3}{2}} u\|_{L^2_t
   L^2_x} \lesssim \|r^{\frac{p-2}{2}}u(0,\cd)\|_{L^2} +
 \|r^{\frac{p-3}{2}} \good (ru)\|_{L^2_t L^2_x}.
  \end{equation}
\end{lemma}

\begin{proof}
 We consider
\[\int_0^T\int r^{p-3} u^2\,dx\,dt = \frac{1}{p-2}\int_0^T\int_{\S^2}\int_0^\infty
  (\partial_t+\partial_r)(r^{p-2}) (ru)^2\,dr\,d\sigma\,dt.\]
Upon integrating by parts, this is
\[=\frac{1}{p-2} \int r^{p-2} u^2(T,x)\,dx -\frac{1}{p-2}\int
  r^{p-2}u^2(0,x)\,dx
+\frac{2}{2-p} \int_0^T\int_{\S^2}\int_0^\infty r^{p-2} (ru)
(\partial_t+\partial_r)(ru)\,dr\,d\sigma\,dt.\]
If we apply the Schwarz inequality to the last term, the above two
equations yield (where the norms in time are taken over $[0,T]$)
\[\|r^{\frac{p-3}{2}}u\|^2_{L^2_tL^2_x} +
  \|r^{\frac{p-2}{2}}u(T,\cd)\|^2_{L^2_x} \lesssim
  \|r^{\frac{p-2}{2}}u(0,\cd)\|^2_{L^2_x} + 
  \|r^{\frac{p-3}{2}}
  u\|_{L^2_tL^2_x}\|r^{\frac{p-3}{2}}(\partial_t+\partial_r)(ru)\|_{L^2_tL^2_x}.\]
Using that $ab\le c a^2 + \frac{1}{4c}b^2$
for any $c>0$, the first factor of the last term can be
absorbed into the left side.  The proof
is then completed by taking a supremum over $T$.
\end{proof}

The next result is the main linear estimate used in our proof of
global existence.  It is based on the $r^p$-weighted local energy
estimates of \cite{DafRod}.  Here we have adapted the proof to allow
for small, time-dependent perturbations of the geometry in order to
accommodate the quasilinear nature of the problem.  Due to the
``problem of multiple good derivatives,'' we do so in two different
ways.  The first estimate, which will be applied with the highest
order of vector fields, uses integration by parts on the
perturbation in the most standard way.  Upon doing so, it is possible
that both good derivatives will land on the perturbation.  To handle
this, the second estimate, which will be used at a lower order and
with a higher $p$, will
be employed.  In this second case, if neither of the derivatives in the quasilinear term are
good derivatives, no further integration by parts will be applied.
This will keep the two good derivatives on separate terms, which will
each be at this lower order.

\begin{theorem}\label{thm:rp}  Suppose $h\in C^2([0,\infty)\times \R^3)$
  satisfies \eqref{h_symmetry}.  Let $u\in C^2([0,\infty) \times
  \R^3)$ be so that for each $t\ge 0$,
\[r^{\frac{p+2}{2}}|\partial^{\le 1} u(t,x)|\to 0
  \text{ as } |x|\to \infty.\]
Then, for any $0<p<2$,
  \begin{multline}
    \label{rp1}
    \|r^{\frac{p-1}{2}} \good u\|^2_{L^2_tL^2_x} + \|
    r^{\frac{p-3}{2}} u\|^2_{L^2_tL^2_x} + 
\|r^{\frac{p}{2}} \good u\|^2_{L^\infty_t L^2_x} + \|r^{\frac{p-2}{2}}
u\|^2_{L^\infty_t L^2_x}
    \lesssim 
\|r^{\frac{p}{2}} \good u(0,\cd)\|^2_{L^2_x} + \|r^{\frac{p-2}{2}}
u(0,\cd)\|^2_{L^2_x}
\\+ \sup_t \Bigl(\int r^p |h| |\partial u|
    \Bigl(|\partial u|+r^{-1}|u|\Bigr)\,dx\Bigr)
\\+ \int_0^\infty \int r^p |\Box_h u| \Bigl(|\good
    u| + \frac{|u|}{r}\Bigr)\,dx\,dt + \int_0^\infty \int r^p |\partial h| |\good u|\Bigl(|\good
u|+\frac{|u|}{r}\Bigr)\,dx\,dt
\\
+ \int_0^\infty \int r^p
\Bigl(\frac{|h|}{r}+|\good h| + |\partial(\omega_\alpha\omega_\beta
h^{\alpha\beta})|\Bigr) \ |\partial u|\Bigl(|\partial
u|+\frac{|u|}{r}\Bigr)\,dx\,dt, 
  \end{multline}
 and 
     \begin{multline}
     \label{rp2}
    \|r^{\frac{p-1}{2}} \good u\|^2_{L^2_tL^2_x} + \|
     r^{\frac{p-3}{2}} u\|^2_{L^2_tL^2_x} + \|r^{\frac{p}{2}} \good u\|^2_{L^\infty_t L^2_x} + \|r^{\frac{p-2}{2}}
 u\|^2_{L^\infty_t L^2_x}
     \lesssim 
 \|r^{\frac{p}{2}} \good u(0,\cd)\|^2_{L^2_x} + \|r^{\frac{p-2}{2}}
 u(0,\cd)\|^2_{L^2_x}
 \\+ \sup_t \Bigl(\int r^p |h| \Bigl(|\partial u|+\frac{|u|}{r}\Bigr)\Bigl(|\good
 u|+\frac{|u|}{r}\Bigr)\,dx\Bigr)
 +\int_0^\infty \int r^p |\Box_h u|\Bigl(|\good
 u|+\frac{|u|}{r}\Bigr)\,dx\,dt
 \\+\int_0^\infty \int r^{p-1} |h| |\partial u| \Bigl(|\good
 u|+\frac{|u|}{r}\Bigr)\,dx\,dt
+\int_0^\infty \int r^p |\partial h| |\good u|\Bigl(|\good
 u|+\frac{|u|}{r}\Bigr)\,dx\,dt
 \\+ \int_0^\infty \int r^p \Bigl(|\good
 h|+|h^{\alpha\beta}\omega_\alpha\omega_\beta|\Bigr)
 |\partial\partial^{\le 1} u|
 \Bigl(|\good u|+\frac{|u|}{r}\Bigr)\,dx\,dt 
+\int_0^\infty\int r^{p-3} \Bigl(r|\partial h|+|h|\Bigr) |u|^2.
  \end{multline}
\end{theorem}

The reader should have in mind that in the sequel we will choose the
perturbation to have the form $h^{I,\alpha\beta}_K =
-B^{I,\gamma\alpha\beta}_{JK}\partial_\gamma u^J$.
We note that the last term in \eqref{rp1} could potentially have
multiple good derivatives when the perturbation $h$ is itself based in
a good
derivative.  As indicated above, to remedy this, we will later use \eqref{rp2} with its $p$
chosen to be more than twice that used in \eqref{rp1}.  We also point
out that the second to last term of \eqref{rp2} has a factor
containing more derivatives than what appear in the left, which
requires that this estimate be applied with a lower number of vector
fields so this loss of regularity can be overcome.

\begin{proof}
  We consider
  \[\int_0^T \int r^p \Box_h u^I
    \Bigl(\partial_t+\partial_r+\frac{1}{r}\Bigr) u^I\,dx\,dt.\]

  To start, we argue as in \cite{FMM} and note that
  \[\int_0^T \int r^p \Box u^I
    \Bigl(\partial_t+\partial_r+\frac{1}{r}\Bigr) u^I\,dx\,dt
=\int_0^T \int r^p
\Bigl[\Bigl(\partial_t-\partial_r\Bigr)\Bigl(\partial_t+\partial_r\Bigr)(ru^I)
- \ang\cdot\ang (ru^I)\Bigr] \Bigl(\partial_t + \partial_r\Bigr)(ru^I)\,dr\,d\sigma\,dt.\]
Integrating by parts and using $[\nabla,\partial_r] = \frac{1}{r}\ang$, we see that the right side is
\begin{multline*}
  =\frac{1}{2}\int_0^T\int
    r^p\Bigl(\partial_t-\partial_r\Bigr)\Bigl|\Bigl(\partial_t+\partial_r\Bigr)(ru)\Bigr|^2\,dr\,d\sigma\,dt
    + \frac{1}{2}\int_0^T\int r^p
    \Bigl(\partial_t+\partial_r\Bigr)|\ang(ru)|^2\,dr\,d\sigma\,dt
    \\+ \int_0^T\int r^{p-1}|\ang(ru)|^2\,dr\,d\sigma\,dt.
  \end{multline*}
  Subsequent integrations by parts give that this is
  \begin{multline*}
  = \frac{1}{2} \|r^{\frac{p-2}{2}}\good (ru)(T,\cd)\|^2_{L^2} -
\frac{1}{2} \|r^{\frac{p-2}{2}}\good (ru)(0,\cd)\|^2_{L^2}
\\+ \frac{p}{2}\int_0^T\int
  r^{p-1}\Bigl|\Bigr(\partial_t+\partial_r\Bigr)(ru)\Bigr|^2\,dr\,d\sigma\,dt
  +\Bigl(1-\frac{p}{2}\Bigr)\int_0^T\int r^{p-1}|\ang(ru)|^2\,dr\,d\sigma\,dt.
  \end{multline*}
Provided that $0<p<2$, we can combine this with \eqref{Hardy} to
obtain
\begin{multline}
  \label{no_pert}
  \|r^{\frac{p-1}{2}} \good u\|^2_{L^2_tL^2_x} + \|r^{\frac{p-3}{2}}
  u\|^2_{L^2_tL^2_x} + \|r^{\frac{p}{2}} \good u(T,\cd)\|^2_{L^2_x} + \|r^{\frac{p-2}{2}}
u(T,\cd)\|^2_{L^2_x} \\\lesssim \|r^{\frac{p}{2}} \good u(0,\cd)\|^2_{L^2_x} + \|r^{\frac{p-2}{2}}
u(0,\cd)\|^2_{L^2_x} +
 \Bigl| \int_0^T \int r^p \Box u^I \Bigl(\partial_t+\partial_r+\frac{1}{r}\Bigr)u^I\,dx\,dt\Bigr|.
\end{multline}

We now consider the perturbation terms.  Using, again, that
$[\nabla,\partial_r] = \frac{1}{r}\ang$ and the symmetries
\eqref{h_symmetry}, we obtain  
   \begin{multline*}\int_0^T\int r^p h^{I,\alpha\beta}_K
    \partial_\alpha\partial_\beta
    u^K \Bigl(\partial_t + \partial_r + \frac{1}{r}\Bigr) u^I\,dx\,dt
=\int r^p h^{I,0\beta}_K \partial_\beta
u^K\Bigl(\partial_t+\partial_r+\frac{1}{r}\Bigr)u^I\,dx\Bigl|_{t=0}^T\\
-\int_0^T \int r^{p-1} h^{I,j\beta}_K \partial_\beta u^K \ang_j
u^I\,dx\,dt
+\int_0^T\int r^{p-2} \omega_j
h^{I,j\beta}_K \partial_\beta u^K u^I\,dx\,dt
\\-\frac{1}{2}\int_0^T \int r^p h^{I,\alpha\beta}_K 
\Bigl(\partial_t+\partial_r+\frac{1}{r}\Bigr)\Bigl[\partial_\beta u^K\partial_\alpha u^I\Bigr]\,dx\,dt
-\int_0^T \int r^p \partial_\alpha h^{I,\alpha\beta}_K \partial_\beta
u^K \Bigl(\partial_t+\partial_r+\frac{1}{r}\Bigr)u^I\,dx\,dt
\\-p\int_0^T \int r^{p-1} \omega_j h^{I,j\beta}_K \partial_\beta u^K\Bigl(\partial_t+\partial_r+\frac{1}{r}\Bigr)u^I\,dx\,dt,
\end{multline*}
where, as above, we have set $\omega = (-1, x/r)$.
And thus, arguing as in \eqref{nullify} and \eqref{nullify2} and
integrating by parts, this is
\begin{multline*}
  =\int r^p h^{I,0\beta}_K \partial_\beta
  u^K\Bigl(\partial_t+\partial_r+\frac{1}{r}\Bigr)u^I\,dx\Bigl|_{t=0}^T
-\frac{1}{2}\int r^p h^{I,\alpha\beta}_K \partial_\beta u^K
  \partial_\alpha u^I\,dx\Bigl|_{t=0}^T
  \\
-\int_0^T \int r^{p-1} h^{I,j\beta}_K \partial_\beta u^K \ang_j u^I\,dx\,dt
+\int_0^T\int r^{p-2} h^{I,j\beta}_K u^I \partial_\beta u^K\,dx\,dt
\\+\frac{p+1}{2}\int_0^T \int r^{p-1} h^{I,\alpha\beta}_K 
\partial_\beta u^K\partial_\alpha u^I\,dx\,dt
+\frac{1}{2}\int_0^T \int r^{p} \Bigl(\partial_t+\partial_r\Bigr)h^{I,\alpha\beta}_K 
\partial_\beta u^K\partial_\alpha u^I\,dx\,dt
\\-\int_0^T \int r^p \partial_\alpha h^{I,\alpha\beta}_K
(\partial_\beta-\omega_\beta \partial_r)
u^K \Bigl(\partial_t+\partial_r+\frac{1}{r}\Bigr)u^I\,dx\,dt
-\int_0^T \int r^p \omega_\beta\Bigl(\partial_\alpha-\omega_\alpha\partial_r\Bigr) h^{I,\alpha\beta}_K \partial_r
u^K \Bigl(\partial_t+\partial_r+\frac{1}{r}\Bigr)u^I\,dx\,dt
\\-\int_0^T \int r^p \omega_\beta\omega_\alpha\partial_r h^{I,\alpha\beta}_K \partial_r
u^K \Bigl(\partial_t+\partial_r+\frac{1}{r}\Bigr)u^I\,dx\,dt
-p\int_0^T \int r^{p-1} \omega_j h^{I,j\beta}_K \partial_\beta u^K\Bigl(\partial_t+\partial_r+\frac{1}{r}\Bigr)u^I\,dx\,dt. 
\end{multline*}
From this, when combined with \eqref{no_pert}, the bound \eqref{rp1}
follows immediately.

We next consider \eqref{rp2}.  We write
  \[h^{I,\alpha\beta}_K \partial_\alpha \partial_\beta u^K
    = h^{I,\alpha\beta}_K\omega_\alpha\omega_\beta \partial_r^2 u^K
    + h^{I,\alpha\beta}_K \partial_\alpha \Bigl(\partial_\beta
    -\omega_\beta\partial_r\Bigr)u^K
    + h^{I,\alpha\beta}_K \Bigl(\partial_\alpha-\omega_\alpha\partial_r\Bigr)
    \omega_\beta \partial_r u^K.
  \]
For \eqref{rp2}, we need not further modify the terms involving
$h^{I,\alpha\beta}_K \omega_\alpha \omega_\beta \partial_r^2 u^K$.
For those that remain, we notice that
  \begin{multline*}\int_0^T\int r^p h^{I,\alpha\beta}_K
    \Bigl(\partial_\alpha\partial_\beta-\omega_\alpha\omega_\beta \partial_r^2\Bigr)
    u^K \Bigl(\partial_t + \partial_r + \frac{1}{r}\Bigr) u^I\,dx\,dt
=\int r^p h^{I,0\beta}_K \partial_\beta
u^K\Bigl(\partial_t+\partial_r+\frac{1}{r}\Bigr)u^I\,dx\Bigl|_{t=0}^T\\
-p\int_0^T \int r^{p-1}\omega_j h^{I,j\beta,}_K \partial_\beta u^K
\Bigl(\partial_t+\partial_r+\frac{1}{r}\Bigr)u^I\,dx\,dt
-\int_0^T\int r^p \partial_\alpha h^{I,\alpha\beta}_K \Bigl(\partial_\beta
-\omega_\beta\partial_r\Bigr)u^K
\Bigl(\partial_t+\partial_r+\frac{1}{r}\Bigr)u^I\,dx\,dt
\\-\int_0^T\int r^p (\partial_\alpha
-\omega_\alpha\partial_r)h^{I,\alpha\beta}_K \omega_\beta \partial_r
u^K \Bigl(\partial_t+\partial_r+\frac{1}{r}\Bigr)u^I\,dx\,dt
\\+(p+2)\int_0^T\int r^{p-1}
h^{I,\alpha\beta}_K\omega_\alpha\omega_\beta \partial_ru^K\Bigl(\partial_t+\partial_r+\frac{1}{r}\Bigr)u^I\,dx\,dt
\\-\int_0^T \int r^{p-1} h^{I,j\beta}_K \partial_\beta u^K \ang_j
u^I\,dx\,dt + \int_0^T\int r^{p-2} h^{I,j\beta}_K \omega_j
u^I \partial_\beta u^K\,dx\,dt
\\-\int_0^T\int r^{p-2} h^{I,\alpha\beta}_K \omega_\alpha\omega_\beta
u^I \partial_r u^K\,dx\,dt
\\-\frac{1}{2}\int_0^T\int r^ph^{I,\alpha\beta}_K
\Bigl(\partial_t+\partial_r+\frac{1}{r}\Bigr)\Bigl[(\partial_\beta -
\omega_\beta\partial_r)u^K \partial_\alpha u^I + \omega_\beta \partial_r u^K(\partial_\alpha-\omega_\alpha\partial_r)u^I\Bigr]\,dx\,dt.
\end{multline*}
A subsequent integration by parts gives that 
\begin{multline*}
  -\frac{1}{2}\int_0^T\int r^ph^{I,\alpha\beta}_K
\Bigl(\partial_t+\partial_r+\frac{1}{r}\Bigr)\Bigl[(\partial_\beta -
\beta\partial_r)u^K \partial_\alpha u^I + \omega_\beta \partial_r
u^K(\partial_\alpha-\omega_\alpha\partial_r)u^I\Bigr]\,dx\,dt
\\=-\frac{1}{2}\int r^p
h^{I,\alpha\beta}_K\Bigl[(\partial_\beta-\omega_\beta \partial_r)u^K\partial_\alpha
u^I + \omega_\beta\partial_r
u^K(\partial_\alpha-\omega_\alpha\partial_r)u^I\Bigr]\,dx\Bigl|_0^T
\\+\frac{1}{2}\int_0^T\int r^p
(\partial_t+\partial_r)h^{I,\alpha\beta}_K\Bigl[(\partial_\beta-\omega_\beta\partial_r)u^K\partial_\alpha
u^I + \omega_\beta\partial_r
u^K(\partial_\alpha-\omega_\alpha\partial_r)u^I\Bigr]\,dx\,dt
\\+\frac{p+1}{2}\int_0^T\int r^{p-1}h^{I,\alpha\beta}_K\Bigl[ (\partial_\beta-\omega_\beta\partial_r)u^K\partial_\alpha
u^I + \omega_\beta\partial_r
u^K(\partial_\alpha-\omega_\alpha\partial_r)u^I\Bigr]\,dx\,dt.
\end{multline*}
Moreover,
\begin{multline*}
\int_0^T\int r^{p-2} h^{I,j\beta}_K \omega_j
u^I \partial_\beta u^K\,dx\,dt  -\int_0^T\int r^{p-2} h^{I,\alpha\beta}_K \omega_\alpha\omega_\beta
u^I \partial_r u^K\,dx\,dt = \frac{1}{2}\int r^{p-2}h^{I,j0}_K u^K
u^I\,dx\Bigl|_0^T \\- \frac{p-1}{2}\int_0^T\int r^{p-3}\omega_j\omega_l
h^{I,jl}_K u^K u^I\,dx\,dt -\frac{1}{2}\int_0^T\int r^{p-3} h^{I,jj}_K
u^K u^I\,dx\,dt - \frac{1}{2}\int_0^T\int
r^{p-2}\omega_j \partial_\beta h^{I,j\beta}_K u^K u^I\,dx\,dt \\+
\frac{p}{2}\int_0^T\int r^{p-3}
h^{I,\alpha\beta}_K\omega_\alpha\omega_\beta u^K u^I\,dx\,dt +
\frac{1}{2}\int_0^T\int r^{p-2}\partial_r h^{I,\alpha\beta}_K
\omega_\alpha\omega_\beta u^K u^I\,dx\,dt.
\end{multline*}
Combining \eqref{no_pert} with the preceding three equations
yields \eqref{rp2}.
\end{proof}

The main source of decay that we rely upon is the following
weighted Sobolev estimate that originates from \cite{klainerman_null}.  It is
proved by localizing and applying standard Sobolev embeddings in the $r, \omega$
variables.  The decay results upon adjusting the volume element
$dr\,d\sigma(\omega)$ to match that of $\R^3$ in spherical coordinates.
\begin{lemma}
  \label{lemma_weighted_Sobolev}
 For $h\in C^\infty(\R^3)$ and $R\ge 1$, we have
  \begin{equation} 
    \label{weighted_Sobolev}
    \|h\|_{L^\infty(R/2<|x|<R)} \lesssim R^{-1} \|Z^{\le 2} h\|_{L^2(R/4<|x|<2R)}.
  \end{equation}
\end{lemma}

We now use the smallness of $h$ to absorb
some perturbative terms and first undertake \eqref{rp1}.  The following proposition
largely addresses the problem of multiple good derivatives.  In the
sequel, the perturbation will be a lower order term.  When this lower
order term is small in a weighted space with the $\tilde{p}$ more than twice
the choice of $p$ for the higher order factors, we are able to absorb
the perturbative factors, including the last term in \eqref{rp1},
which is the possible occurrence of multiple good derivatives.  The
resulting estimate is then quite similar to that used in \cite{FMM}
for the semilinear case.  The issue with multiple good derivatives
barely appears in the next section as it is entirely reduced to
demonstrating hypothesis \eqref{h_assumptions} below.
\begin{proposition}\label{prop_high}  Fix $0<p<1$.
Assume that $h\in C^2([0,\infty)\times \R^3)$ satisfies
\eqref{h_symmetry}.  Moreover, for $\tilde{p}>2p$, suppose
\begin{equation}
  \label{h_assumptions}
  \|Z^{\le 3} h\|_{L^\infty_tL^2_x}
+ \|\la r\ra^{\frac{\tilde{p}-1}{2}} Z^{\le 2}
  \good h\|_{L^2_tL^2_x}+
  \|\la r\ra^{\frac{\p-1}{2}}Z^{\le 3}(\omega_\alpha\omega_\beta h^{\alpha\beta})\|_{L^2_t
    L^2_x}
\le \delta
\end{equation}
for $\delta>0$ sufficiently small.  Let $u\in
C^2([0,\infty)\times\R^3)$ be so that for each $t\ge 0$,
\[r^{\frac{p+2}{2}}|\partial^{\le 1} Z^{\le N} u(t,x)|\to 0,\quad
  \text{ as } |x|\to \infty.\]
Then
\begin{multline}
  \label{high}
\|\la r\ra^{\frac{p}{2}}\good Z^{\le N} u\|_{L^\infty_tL^2_x} +
\|\la r\ra^{\frac{p-2}{2}} Z^{\le N} u\|_{L^\infty_t L^2_x} +
\|\partial Z^{\le N} u\|_{L^\infty_t L^2_x}
\\+
  \|\la r\ra^{\frac{p-1}{2}} \good Z^{\le N} u\|_{L^2_tL^2_x} + \|\la
  r\ra^{\frac{p-3}{2}} Z^{\le N} u\|_{L^2_tL^2_x} + \|Z^{\le N}
  u\|_{LE^1}
\\\lesssim \|\la r\ra^{\frac{p}{2}}\good Z^{\le N} u(0,\cd)\|_{L^2_x} + \|\la
 r\ra^{\frac{p-2}{2}} Z^{\le N} u(0,\cd)\|_{L^2_x} + \|\partial Z^{\le
   N} u(0,\cd)\|_{L^2_x} 
+\|\la
r\ra^{\frac{p+1}{2}} \Box_h Z^{\le N} u\|_{L^2_tL^2_x}.
\end{multline}
\end{proposition}

\begin{proof}
  We first apply \eqref{rp1} to $\chi_{>1}(|x|) Z^{\le N} u$ where
  $\chi_{>1}(r)$ is a smooth function that vanishes for $r\le 1$ and is
  identically $1$ for $r>2$.  Using
  a Sobolev embedding, $H^2(\R^3)\subseteq L^\infty(\R^3)$, and the
  bound for the first term in \eqref{h_assumptions}, it follows that 
\[\int_0^\infty\int r^p |[\Box_h,\chi_{>1}] Z^{\le N} u||\good
  (\chi_{>1} Z^{\le N} u)|\,dx\,dt\lesssim (1+\|\partial^{\le 1}
  h\|_{L^\infty_t L^\infty_x}) \|Z^{\le N} u\|^2_{LE^1} \lesssim
  \|Z^{\le N} u\|^2_{LE^1}.\]
Thus by subsequently applying \eqref{LEnull} to $Z^{\le N} u$, we see that
the square of the left side of \eqref{high} is bounded by
\begin{multline}
  \label{high_1}
 \|\la r\ra^{\frac{p}{2}}\good Z^{\le N} u(0,\cd)\|^2_{L^2_x} + \|\la
 r\ra^{\frac{p-2}{2}} Z^{\le N} u(0,\cd)\|^2_{L^2_x} + \|\partial Z^{\le
   N} u(0,\cd)\|^2_{L^2_x} 
\\+ \sup_t \int \la r\ra^p |h| 
\Bigl(|\partial Z^{\le N} u| + \frac{|Z^{\le N} u|}{\la r\ra}\Bigr)^2\,dx
\\
+ \int_0^\infty \int
    \Bigl(\la r\ra^p |\good Z^{\le N} u|+|\partial Z^{\le N} u| + \la
    r\ra^{p-1} |Z^{\le N} u|\Bigr)
    |\Box_h Z^{\le N} u|\,dx\,dt
\\+\int_0^\infty \int \la r\ra^p |\partial h|
\Bigl(|\good Z^{\le N} u| + \frac{|Z^{\le N} u|}{\la
  r\ra}\Bigr)^2\,dx\,dt 
+\int_0^\infty \int |\parital h| |\good Z^{\le N} u| \Bigl(|\partial
Z^{\le N}u| + \frac{|Z^{\le N}u|}{r}\Bigr)\,dx\,dt
\\
+\int_0^\infty \int \la r\ra^p \Bigl(\frac{|h|}{\la r\ra} + |\good h|
+ |\partial(\omega_\alpha\omega_\beta
h^{\alpha\beta})|\Bigr)\Bigl(|\partial Z^{\le N} u| + \frac{|Z^{\le
    N}u|}{r}\Bigr)^2\,dx\,dt.
\end{multline}

By the Cauchy-Schwarz inequality and the fact that $p>0$, we obtain
\begin{multline*}
  \int_0^\infty \int
    \Bigl(\la r\ra^p |\good Z^{\le N} u|+|\partial Z^{\le N} u| + \la
    r\ra^{p-1} |Z^{\le N} u|\Bigr)
    |\Box_h Z^{\le N} u|\,dx\,dt
\\\le \frac{1}{2}\Bigl(\|\la r\ra^{\frac{p-1}{2}}\good
Z^{\le N} u\|^2_{L^2_tL^2_x} + \|Z^{\le N} u\|_{LE^1}^2 + \|\la
r\ra^{\frac{p-3}{2}}Z^{\le N} u\|^2_{L^2_tL^2_x}\Bigr) + C \|\la
r\ra^{\frac{p+1}{2}} \Box_h Z^{\le N} u\|_{L^2_tL^2_x}^2.
\end{multline*}
The first three terms in the right side can be absorbed by the square
of the left
side of \eqref{high}..

We will proceed to showing that the fourth, sixth, seventh, and eighth
terms of \eqref{high_1} can be bounded by a constant that can be
chosen sufficiently small times the square of the left side of
\eqref{high}.  These terms can again be absorbed, which will complete
the argument.

Since $p<1$, using \eqref{weighted_Sobolev}, a standard Hardy inequality
\begin{equation}\label{xhardy}
 \|r^{-1} u\|_{L^2(\R^3)} \lesssim \|\nabla
u\|_{L^2(\R^3)},
\end{equation}
and \eqref{h_assumptions} results in
\[\sup_t \int \la r\ra^p |h|\Bigl(|\partial Z^{\le N} u| + \frac{|Z^{\le N}
    u|}{\la r\ra}\Bigr)^2\,dx \lesssim \delta \|\partial Z^{\le N}
  u\|^2_{L^\infty_t L^2_x},\]
which suffices for the fourth term in \eqref{high_1}.

Proceeding to the sixth term in \eqref{high_1}, we apply
\eqref{weighted_Sobolev} and \eqref{h_assumptions} to get
\[\int_0^\infty \int \la r\ra^p |\partial h|\Bigl(|\good Z^{\le N}
  u|+\frac{|Z^{\le N} u|}{\la r\ra}\Bigr)^2 \,dx\,dt \lesssim\delta
  \Bigl(\|\la r\ra^{\frac{p-1}{2}} \good Z^{\le N}u\|^2_{L^2_tL^2_x} +
  \|\la r\ra^{\frac{p-3}{2}} Z^{\le N} u\|^2_{L^2_tL^2_x}\Bigr).\]
Similarly, since $p>0$,
\[\int_0^\infty \int |\partial h| |\good Z^{\le N}u|\Bigl(|\partial
  Z^{\le N} u| + \frac{|Z^{\le N}u|}{r}\Bigr)\,dx\,dt
\lesssim \delta \|\la r\ra^{\frac{p-1}{2}} \good Z^{\le N}
u\|_{L^2_tL^2_x} \|Z^{\le N} u\|_{LE^1}.\]
And since $p<1$, \eqref{weighted_Sobolev} and \eqref{h_assumptions} give
\[\int_0^\infty \int \la r\ra^{p-1} |h| \Bigl(|\partial Z^{\le N} u| + \frac{|Z^{\le
    N}u|}{\la r\ra}\Bigr)^2\,dx\,dt
\lesssim \delta \|Z^{\le N} u\|_{LE^1}^2.
\]

It remains to establish
\begin{equation}\label{prop25lastpiece}\int_0^\infty \int \la r\ra^p \Bigl(|\good h| +
  |\partial(\omega_\alpha\omega_\beta
  h^{\alpha\beta})|\Bigr)\Bigl(|\partial Z^{\le N} u|+\frac{|Z^{\le
      N}u|}{r}\Bigr)^2\,dx\,dt
\lesssim \delta \|\partial Z^{\le N} u\|_{L^\infty_tL^2_x} \|Z^{\le
  N}u\|_{LE^1}.\end{equation}
The left side of \eqref{prop25lastpiece} is bounded by
\[\sum_{j\ge 0} 2^{pj} \int_0^\infty \int_{\{2^{j-1}\le \la x \ra\le
    2^j\}} \Bigl(|\good h| + |\partial(\omega_\alpha\omega_\beta
  h^{\alpha\beta})|\Bigr) \Bigl(|\partial Z^{\le N} u|+\frac{|Z^{\le
      N} u|}{r}\Bigr)^2\,dx\,dt.\]
Applying \eqref{weighted_Sobolev}, the Schwarz inequality, and the
Hardy inequality \eqref{xhardy},  this is
controlled by
\[
 \Bigl(\sum_{j\ge 0} 2^{j(p-\frac{\p}{2})}\Bigr)  \Bigl(\|\la r\ra^{\frac{\p-1}{2}} Z^{\le 2}\good
  h\|_{L^2_tL^2_x} +
  \|\la r\ra^{\frac{\p-1}{2}} Z^{\le 2}
  \partial(\omega_\alpha\omega_\beta
  h^{\alpha\beta})\|_{L^2_tL^2_x}\Bigr) \|\partial Z^{\le N} u\|_{L^\infty_t L^2_x} \|Z^{\le N} u\|_{LE^1}.
\]
Using the bound on the last two terms of
\eqref{h_assumptions}, as $\p>2p$, \eqref{prop25lastpiece} follows.
\end{proof}

We next consider a result analogous to Proposition~\ref{prop_high} for
the lower order energy, which has the larger weight $\tilde{p}$.  The
proof proceeds similarly but is based instead on \eqref{rp2}.  It is
this estimate that will allow us to show \eqref{h_assumptions} in the
sequel, which thus addresses the issue with multiple good
derivatives.  We note that the following will be applied at a lower order
than \eqref{high}, and as such, we will be able to handle the loss of
a vector field that occurs in the last two terms of the estimate.
\begin{proposition}\label{prop_low}  Fix $0<p<1$.
Assume that $h\in C^2([0,\infty)\times \R^3)$ satisfies
\eqref{h_symmetry} and, for $2p<\tilde{p}<2$, \eqref{h_assumptions}
with $\delta>0$ sufficiently small.  Let $u\in
C^2([0,\infty)\times\R^3)$ be so that for each $t\ge 0$,
\[r^{\frac{p+2}{2}}|\partial^{\le 1} Z^{\le N} u(t,x)|\to 0,\quad
  \text{ as } |x|\to \infty.\]
Then
\begin{multline}
  \label{low}
\|\la r\ra^{\frac{\tilde{p}}{2}}\good Z^{\le N-1} u\|_{L^\infty_tL^2_x} +
\|\la r\ra^{\frac{\tilde{p}-2}{2}} Z^{\le N-1} u\|_{L^\infty_t L^2_x} +
  \|\la r\ra^{\frac{\tilde{p}-1}{2}} \good Z^{\le N-1} u\|_{L^2_tL^2_x} + \|\la
  r\ra^{\frac{\tilde{p}-3}{2}} Z^{\le N-1} u\|_{L^2_tL^2_x} 
\\\lesssim \|\la r\ra^{\frac{\p}{2}}\good Z^{\le N-1} u(0,\cd)\|_{L^2_x} + \|\la
 r\ra^{\frac{\p-2}{2}} Z^{\le N-1} u(0,\cd)\|_{L^2_x} 
+\|\la
r\ra^{\frac{\p+1}{2}} \Box_h Z^{\le N-1} u\|_{L^2_tL^2_x}
\\+ \|Z^{\le N} u\|_{LE^1} + \|\partial Z^{\le N} u\|_{L^\infty_tL^2_x}.
\end{multline}
\end{proposition}

\begin{proof}
  As in the preceding proof,  \eqref{rp2} can be
applied to control the square of the left side of \eqref{low} by
  \begin{multline*}
    \|\la r\ra^{\frac{\p}{2}} \good Z^{\le N-1} u(0,\cd)\|^2_{L^2_x} +
    \|\la r\ra^{\frac{\p-2}{2}} Z^{\le N-1} u(0,\cd)\|^2_{L^2_x}
\\+ \sup_t \Bigl(\int \la r\ra^{\p} |h| \Bigl(|\partial Z^{\le N-1}u|+
\frac{|Z^{\le N-1}u|}{\la r\ra}\Bigr)
\Bigl(|\good Z^{\le N-1} u|+ \frac{|Z^{\le N-1} u|}{\la
  r\ra}\Bigr)\,dx\Bigr)
\\
+\int_0^\infty\int \la r\ra^{\p} |\Box_h Z^{\le N-1} u|\Bigl(|\good
Z^{\le N-1} u| + \frac{|Z^{\le N-1}u|}{\la r\ra}\Bigr)\,dx\,dt
\\+ \int_0^\infty \int \la r\ra^{\p-1} |h| |\partial Z^{\le N-1}
u|\Bigl(|\good Z^{\le N-1} u|+\frac{|Z^{\le N-1}u|}{\la r\ra}\Bigr)\,dx\,dt
\\+ \int_0^\infty\int \la r\ra^\p |\partial h| |\good Z^{\le N-1}
u|\Bigl(|\good Z^{\le N-1} u|+\frac{|Z^{\le N-1} u|}{\la
  r\ra}\Bigr)\,dx\,dt
\\+\int_0^\infty \int \la r\ra^\p \Bigl(|\good h| +
|h^{\alpha\beta}\omega_\alpha\omega_\beta|\Bigr) |\partial Z^{\le N}
u| \Bigl(|\good Z^{\le N-1} u| + \frac{|Z^{\le N-1} u|}{\la r\ra}\Bigr)\,dx\,dt
\\+ \int_0^\infty \int \la r\ra^{\p-3}\Bigl(r|\partial
h|+|h|\Bigr)|Z^{\le N-1} u|^2\,dx\,dt+\|Z^{\le N-1} u\|^2_{LE^1} + \|\partial Z^{\le N-1} u\|^2_{L^\infty_tL^2_x}.
  \end{multline*}

  Using the Schwarz inequality, we see that
\begin{multline*}\int_0^\infty\int \la r\ra^\p |\Box_h Z^{\le N-1} u|\Bigl(|\good
  Z^{\le N-1} u|+\frac{|Z^{\le N-1} u|}{\la r\ra}\Bigr)\,dx\,dt
\\\le \|\la r\ra^{\frac{\p+1}{2}} \Box_h Z^{\le N-1}
u\|_{L^2_tL^2_x}\Bigl(\|\la r\ra^{\frac{\p-1}{2}} \good Z^{\le N-1}
u\|_{L^2_tL^2_x} + \|\la r\ra^{\frac{\p-3}{2}} Z^{\le N-1}
u\|_{L^2_tL^2_x}\Bigr)
\end{multline*}
and the second factor can be absorbed by the square of the left side
of \eqref{low} after applying
$ab\le ca^2 + \frac{1}{4c}b^2$.

As above, we now seek to control the third, fifth, sixth, seventh, and
eighth terms by a small parameter times the square of the left side of
\eqref{low}.  These terms can then be absorbed for a sufficiently small
choice of the parameter, which will complete the proof.
  
We first note that
\begin{align*}\sup_t \int \la r\ra^{\p} |h| \Bigl(|\partial &Z^{\le
    N-1} u|+\frac{|Z^{\le N-1}u|}{\la r\ra}\Bigr) \Bigl(|\good Z^{\le
    N-1} u| + \frac{|Z^{\le N-1} u|}{\la r\ra}\Bigr)\,dx\\
  &\lesssim
\|Z^{\le 2} h\|_{L^\infty_tL^2_x} \|\partial Z^{\le N-1}
    u\|_{L^\infty_t L^2_x} \Bigl(\|\la r\ra^{\frac{\p}{2}} \good
    Z^{\le N-1} u\|_{L^\infty_t L^2_x} + \|\la r\ra^{\frac{\p-2}{2}}
    Z^{\le N-1} u\|_{L^\infty_t L^2_x}\Bigr)
  \\&\lesssim \delta \Bigl(\|\partial Z^{\le N-1} u\|^2_{L^\infty_t L^2_x}
+\|\la r\ra^{\frac{\p}{2}} \good Z^{\le N-1} u\|^2_{L^\infty_t
  L^2_x} + \|\la r\ra^{\frac{\p-2}{2}} Z^{\le N-1} u\|^2_{L^\infty_tL^2_x}\Bigr)
\end{align*}
where we have used \eqref{weighted_Sobolev}, \eqref{xhardy},
\eqref{h_assumptions}, and the assumption that $\p<2$.

For the remaining terms, we repeatedly use 
the hypothesis $\p<2$, \eqref{weighted_Sobolev}, and
\eqref{h_assumptions} and obtain the bounds:
\begin{multline*}
  \int_0^\infty \int \la r\ra^{\p-1} |h| |\partial Z^{\le N-1}
  u|\Bigl(|\good Z^{\le N-1} u|+\frac{|Z^{\le N-1} u|}{\la
    r\ra}\Bigr)\,dx\,dt
\\\lesssim \delta \|Z^{\le N-1} u\|_{LE^1} \Bigl(\|\la
r\ra^{\frac{\p-1}{2}} \good Z^{\le N-1} u\|_{L^2_tL^2_x} + \|\la
r\ra^{\frac{\p-3}{2}} Z^{\le N-1} u\|_{L^2_tL^2_x}\Bigr),
\end{multline*}
\begin{multline*}
  \int_0^\infty \int \la r\ra^\p |\partial h| |\good Z^{\le N-1}
  u|\Bigl(|\good Z^{\le N-1} u| + \frac{|Z^{\le N-1} u|}{\la
    r\ra}\Bigr)\,dx\,dt
\\\lesssim \delta \|\la r\ra^{\frac{\p-1}{2}} \good Z^{\le N-1}
u\|_{L^2_tL^2_x} \Bigl(\|\la
r\ra^{\frac{\p-1}{2}} \good Z^{\le N-1} u\|_{L^2_tL^2_x} + \|\la
r\ra^{\frac{\p-3}{2}} Z^{\le N-1} u\|_{L^2_tL^2_x}\Bigr),
\end{multline*}
and
\[  \int_0^\infty\int \la r\ra^{\p-3} \Bigl(r|\partial h| + |h|\Bigr)
  |Z^{\le N-1} u|^2\,dx\,dt \lesssim \delta \|\la
  r\ra^{\frac{\p-3}{2}} Z^{\le N-1} u\|^2_{L^2_tL^2_x},
\]
as desired.  In all three cases, the $\delta$ appears as a result of
the bound on the first term in \eqref{h_assumptions}.
Arguing similarly to \eqref{prop25lastpiece}, using the
bound on the last two terms of
\eqref{h_assumptions}, results in
\begin{multline*}
\int_0^\infty\int \la r\ra^{\p} \Bigl(|\good
h|+|h^{\alpha\beta}\omega_\alpha \omega_\beta|\Bigr) |\partial Z^{\le
  N} u| \Bigl(|\good Z^{\le N-1} u|+\frac{|Z^{\le N-1} u|}{\la
    r\ra}\Bigr)\,dx\,dt \\\lesssim \|\partial Z^{\le N}
  u\|^2_{L^\infty_t L^2_x}+
\delta\Bigl(\|\la
r\ra^{\frac{\p-1}{2}} \good Z^{\le N-1} u\|^2_{L^2_tL^2_x} + \|\la
r\ra^{\frac{\p-3}{2}} Z^{\le N-1} u\|^2_{L^2_tL^2_x}\Bigr).
\end{multline*}

Combining these bounds immediately gives \eqref{low}.
\end{proof}

\section{Proof of Theorem~\ref{thm:main}}
We begin by establishing the following lemma concerning the
interaction of the admissible vector fields with the null condition.
Variants of this lemma are commonplace.
\begin{lemma}
  Suppose that $A^{I,\alpha\beta}_{JK}$ and
  $B^{I,\gamma\alpha\beta}_{JK}$ satisfy \eqref{null_condition}.
  Then, on $|x|\ge 1$, 
  \begin{multline}
    \label{A_null}
    |Z^{\le N} (A^{I,\alpha\beta}_{JK} \partial_\alpha u^J
    \partial_\beta v^K)| \lesssim |Z^{\le N} \good u| | Z^{\le
      N/2} \partial v| + |Z^{\le N/2}\good u| |Z^{\le N}
    \partial v| \\+
    |Z^{\le N} \partial u| |Z^{\le N/2} \good
    v| + |Z^{\le
      N/2} \partial u| |Z^{\le N} \good v|,
  \end{multline}
and
 \begin{multline}
    \label{B_null2}
    |Z^{\le N} (B^{I,\gamma\alpha\beta}_{JK} \partial_\gamma u^J
  \partial_\alpha\partial_\beta v^K)|
  \lesssim |Z^{\le N} \good u| |Z^{\le
    N/2+1}\partial v| + |
  Z^{\le N/2} \good u| |Z^{\le N+1} \partial v| \\+ |Z^{\le N} \partial u|
 \Bigl( |Z^{\le N/2+1} \good v| + r^{-1} |Z^{\le N/2} \partial
 v|\Bigr) + |Z^{\le N/2} \partial
  u|\Bigl(|Z^{\le N+1} \good
  v| + r^{-1} |\partial Z^{\le N} v|\Bigr).
  \end{multline}
for any $N$.  Moreover,
for any multi-index $\mu$ with $|\mu|\le N$,
  \begin{multline}
    \label{B_null}
    |Z^\mu (B^{I,\gamma\alpha\beta}_{JK} \partial_\gamma u^J
  \partial_\alpha\partial_\beta v^K) - B^{I,\gamma\alpha\beta}_{JK}
  \partial_\gamma u^J \partial_\alpha \partial_\beta Z^\mu v^K|
  \lesssim |Z^{\le N} \good u| |Z^{\le
    N/2+1} \partial v| + |Z^{\le N/2} \good u| |Z^{\le N}\partial  v| \\+ |Z^{\le N} \partial u|
  \Bigl(|Z^{\le N/2+1} \good v| + r^{-1} |Z^{\le N/2} \partial v|\Bigr) + |Z^{\le N/2} \partial
  u|\Bigl(|Z^{\le N} \good
  v| + r^{-1} |Z^{\le N}\partial  v|\Bigr).
  \end{multline}
\end{lemma}

\begin{proof}
  By \eqref{null_condition}, we have
  \[A^{I,\alpha\beta}_{JK} \partial_\alpha u^J \partial_\beta v^K =
A^{I,\alpha\beta}_{JK} (\partial_\alpha-\omega_\alpha\partial_r)u^J
\partial_\beta v^K + A^{I,\alpha\beta}_{JK} \omega_\alpha \partial_r
u^J (\partial_\beta-\omega_\beta\partial_r)v^K.
\]
The result \eqref{A_null} then follows from the product rule.

We write
\begin{multline*}B^{I,\gamma\alpha\beta}_{JK} \partial_\gamma u^J
  \partial_\alpha\partial_\beta v^K  = B^{I,\gamma\alpha\beta}_{JK}
  (\partial_\gamma-\omega_\gamma \partial_r) u^J
  \partial_\alpha\partial_\beta v^K + B^{I,\gamma\alpha\beta}_{JK} \omega_\gamma\partial_r u^J
  (\partial_\alpha-\omega_\alpha \partial_r)\partial_\beta v^K\\+
  B^{I,\gamma\alpha\beta}_{JK} \omega_\gamma \omega_\alpha \partial_r
  u^J \partial_r (\partial_\beta-\omega_\beta \partial_r) v^K.
\end{multline*}
We further note that
\begin{multline*}
  B^{I,\gamma\alpha\beta}_{JK} \omega_\gamma\partial_r u^J
  (\partial_\alpha-\omega_\alpha \partial_r)\partial_\beta v^K = B^{I,\gamma\alpha\beta}_{JK} \omega_\gamma\partial_r u^J
  \partial_\beta (\partial_\alpha-\omega_\alpha \partial_r) v^K +
  \frac{1}{r} B^{I,\gamma\alpha j}_{JK} \omega_\gamma\partial_r u^J
  \omega_\alpha \ang_j v^K
  \\+ \frac{1}{r}B^{I,\gamma j k}_{JK}\omega_\gamma \partial_r u^J
  (\delta_{jk}-\omega_j\omega_k)\partial_r v^K.
\end{multline*}
The desired results \eqref{B_null2} and \eqref{B_null} are again a result of the product
rule.
\end{proof}

We solve \eqref{main_equation} by considering an iteration where
$u_0\equiv 0$ and $u_k$ solves
\begin{equation}
  \label{iteration}
\begin{split}  \Box u^I_k& = A^{I,\alpha\beta}_{JK} \partial_\alpha u_{k-1}^J \partial_\beta
  u_{k-1}^K + B^{I,\gamma\alpha\beta}_{JK} \partial_\gamma u_{k-1}^J
  \partial_\alpha\partial_\beta u_k^K,\\
  u_k&(0,\cd)=f,\quad \partial_tu_k(0,\cd)=g.
\end{split}
\end{equation}
We let
$0<2p<\tilde{p}<2$.

{\em Boundedness:}  Our first task will be to show  uniform
boundedness of the sequence $(u_k)$.  We set
  \[
  M_k = \|\la r\ra^{\frac{p-1}{2}}  Z^{\le N} \good
  u_k\|_{L^2_tL^2_x}
+
\|\la r\ra^{\frac{\p-1}{2}} Z^{\le N-1} \good u_k\|_{L^2_tL^2_x}
 + \|Z^{\le N} \partial u_k\|_{L^\infty_tL^2_x} + \|Z^{\le N} u_k\|_{LE^1}.
\]

For $k=1$, we may take $h^{I,\alpha\beta}_{JK}\equiv 0$.  From
\eqref{high} and \eqref{low}, which are being applied to a homogeneous
equation, and \eqref{smalldata}, there exists a constant $C_0$ so that
\[M_1 \le C_0 \varepsilon.\]
We will use induction to show that
\begin{equation}\label{M_goal}M_k \le 2C_0\varepsilon,\quad \text{ for
    all $k\in\N$}.
\end{equation}

Assuming that the bound holds at the $(k-1)$st level, we
set $h^{I,\alpha\beta}_K =
-B^{I,\gamma\alpha\beta}_{JK} \partial_\gamma u^J_{k-1}$.
By \eqref{null_condition}, for any $I,J,K$, we have
\[h^{I,\alpha\beta}_{JK}\omega_\alpha\omega_\beta =
  -B^{I,\gamma\alpha\beta}_{JK}\omega_\alpha\omega_\beta\partial_\gamma u_{k-1}
  =-B^{I,\gamma\alpha\beta}_{JK}\omega_\alpha\omega_\beta(\partial_\gamma-\omega_\gamma\partial_r)u_{k-1}.\]
Thus, since $\p<2$ and using \eqref{commutator},
\begin{multline*}
  \|Z^{\le 3} h\|_{L^\infty_t L^2_x} +\|\la r\ra^{\frac{\p-1}{2}}
  Z^{\le 2}\good h\|_{L^2_tL^2_x} + \|\la r\ra^{\frac{\p-1}{2}}Z^{\le
    3}(\omega_\alpha\omega_\beta h^{\alpha\beta})\|_{L^2_tL^2_x}
    \\\lesssim
    \|Z^{\le 3} \partial u_{k-1}\|_{L^\infty_tL^2_x} + \|\la
    r\ra^{\frac{\p-1}{2}} Z^{\le 3} \good u_{k-1}\|_{L^2_tL^2_x} + \|Z^{\le 3} u_{k-1}\|_{LE^1}.
  \end{multline*}
  By the inductive hypothesis, this is $\mathcal{O}(\varepsilon)$,
  which establishes \eqref{h_assumptions}.
%
%
  Thus by \eqref{high} and
\eqref{low} it will suffice to establish
\begin{equation}
  \label{M_goal2}
  \|\la r\ra^{\frac{p+1}{2}}\Box_h Z^{\le N} u_k\|_{L^2_tL^2_x} +
\|\la r\ra^{\frac{\p+1}{2}} \Box_h Z^{\le N-1} u_k \|_{L^2_tL^2_x}
\lesssim M_{k-1}^2+M_{k-1}M_k.
\end{equation}
As noted previously the problem of multiple good derivatives is only an
artifact of considering estimates for perturbations of $\Box$.  With
\eqref{h_assumptions} established, Proposition~\ref{prop_high}
completely addresses the issue, and the subsequent argument is quite
reminiscent of the simpler semilinear case.

We first notice that
\[\Box_h Z^{\mu} u_k = Z^{\mu}(A^{I,\alpha\beta}_{JK} \partial_\alpha
  u_{k-1}^J \partial_\beta u_{k-1}^K) +
  Z^{\mu}(B^{I,\gamma\alpha\beta}_{JK}\partial_\gamma u_{k-1}^J
  \partial_\alpha\partial_\beta u_k^K) - B^{I,\gamma\alpha\beta}_{JK}
  \partial_\gamma u_{k-1}^J \partial_\alpha\partial_\beta
  Z^{\mu}u_k^K.\]
By \eqref{A_null} and \eqref{B_null}, it follows that
\begin{multline}
  \label{pBox}
  \|\la r\ra^{\frac{p+1}{2}} \Box_h Z^{\le N}u_k\|_{L^2_tL^2_x}
  \lesssim \|\la r\ra^{\frac{p+1}{2}} |Z^{\le N} \good
  u_{k-1}| |Z^{\le \frac{N}{2}} \partial u_{k-1}|\|_{L^2_tL^2_x}
  \\+ \|\la r\ra^{\frac{p+1}{2}} |Z^{\le \frac{N}{2}}\good
  u_{k-1}| | Z^{\le N} \partial u_{k-1}|\|_{L^2_tL^2_x}
 +\|\la r\ra^{\frac{p+1}{2}} |Z^{\le N} \good u_{k-1}|
 |Z^{\le \frac{N}{2}+1}\partial  u_k|\|_{L^2_tL^2_x}
 \\+ \|\la r\ra^{\frac{p+1}{2}} |Z^{\le \frac{N}{2}} \good
 u_{k-1}| |Z^{\le N}\partial u_k|\|_{L^2_tL^2_x}
 + \|\la r\ra^{\frac{p+1}{2}} |Z^{\le N} \partial  u_{k-1}|
 |Z^{\le \frac{N}{2}+1} \good u_k|\|_{L^2_tL^2_x}
 \\+\|\la r\ra^{\frac{p+1}{2}} |Z^{\le \frac{N}{2}} \partial u_{k-1}|
 |Z^{\le N} \good u_k|\|_{L^2_tL^2_x}
+ \|\la r\ra^{\frac{p-1}{2}} |Z^{\le \frac{N}{2}} \partial  u_{k-1}|
 |Z^{\le N} \partial u_{k-1}|\|_{L^2_tL^2_x}
 \\+\|\la r\ra^{\frac{p-1}{2}} | Z^{\le N} \partial u_{k-1}| | 
 Z^{\le \frac{N}{2}+1} \partial u_k|\|_{L^2_tL^2_x}
 + \|\la
 r\ra^{\frac{p-1}{2}} |Z^{\le \frac{N}{2}} \partial  u_{k-1}| |
 Z^{\le N} \partial  u_k|\|_{L^2_tL^2_x}.
\end{multline}
We notice that the last three terms provide the appropriate bounds
when $|x|\le 1$.  For each of the first six terms in the right, we
apply \eqref{weighted_Sobolev} to the term with fewer vector
fields and measure the good derivative factor in a
weighted $L^2_tL^2_x$-space and the other factor in an energy space
$L^\infty_tL^2_x$.  Provided $\frac{N}{2}+3\le N$, by
\eqref{weighted_Sobolev}, we have
\begin{equation}\label{block1}\|\la r\ra^{\frac{p+1}{2}} Z^{\le N} w \,Z^{\le \frac{N}{2}+1}
  v\|_{L^2_tL^2_x}
+\|\la r\ra^{\frac{p+1}{2}} Z^{\le \frac{N}{2}+1} w \, Z^{\le N}
  v\|_{L^2_tL^2_x}
  \lesssim \|\la r\ra^{\frac{p-1}{2}} Z^{\le N} w\|_{L^2_tL^2_x}
  \|Z^{\le N} v\|_{L^\infty_tL^2_x}\end{equation}
and, since $p<2$,
\begin{equation}\label{block2}\|\la r\ra^{\frac{p-1}{2}} Z^{\le N} w\, Z^{\le \frac{N}{2}+1}
  v\|_{L^2_tL^2_x}
  \lesssim \|Z^{\le N} w\|_{L^\infty_tL^2_x} \|Z^{\le N} v\|_{LE}.\end{equation}
Indeed, \eqref{block2} follows as the square of the left side is
controlled by
\[\sum_{j\ge 0} 2^{(p-1)j} \|Z^{\le N} w\, Z^{\le \frac{N}{2}+1}
  v\|^2_{L^2_tL^2_x(\R_+\times \{2^{j-1}\le\la x\ra\le 2^j\})},\]
which after an application of \eqref{weighted_Sobolev} is
\[
  \lesssim \Bigl(\sum_{j\ge 0} 2^{p-2}\Bigr) \|Z^{\le N}
  w\|^2_{L^\infty_t L^2_x} \sup_{j\ge 0} 2^{-j} \|Z^{\le
    \frac{N}{2}+3} v\|^2_{L^2_tL^2_x(\R_+\times \{2^{j-2}\le\la x\ra\le 2^{j+1}\})},
\]
from which \eqref{block2} follows readily.  Using \eqref{block1} and
\eqref{block2} repeatedly, it follows that the right side of
\eqref{pBox} is
\[\lesssim M_{k-1}^2 + M_{k-1}M_k.\]

This provides the bound for the first term in the left side of
\eqref{M_goal2}.  The bound for the second term is nearly identical
where all of the $p$ are replaced by $\tilde{p}$ and the $N$ by
$N-1$.

{\em Convergence:}  We now establish that the sequence $(u_k)$
is Cauchy.  It, thus, converges, and its limit is the desired
solution.

To this end, we set
\begin{equation}
\label{Ak}
  A_k = \|\la r\ra^{\frac{p-1}{2}} Z^{\le
    N-1}\good (u_k-u_{k-1})\|_{L^2_tL^2_x}
  + \|Z^{\le
    N-1}\partial (u_k-u_{k-1})\|_{L^\infty_tL^2_x} + \|Z^{\le
    N-1}(u_k-u_{k-1})\|_{LE^1}.
\end{equation}
We will prove that
\begin{equation}
  \label{A_goal}
  A_k \le \frac{1}{2}A_{k-1} \quad \text{ for all $k$.}
\end{equation}

We begin by noting
\begin{multline*}\Box (u_k^I-u_{k-1}^I) =
  A^{I,\alpha\beta}_{JK}\partial_\alpha (u_{k-1}^J - u_{k-2}^J)\partial_\beta
  u_{k-1}^K + A^{I,\alpha\beta}_{JK} \partial_\alpha u_{k-2}^J
  \partial_\beta (u_{k-1}^K-
  u_{k-2}^K )
\\+B^{I,\gamma\alpha\beta}_{JK} \partial_\gamma u_{k-1}^J
\partial_\alpha\partial_\beta (u_k^K-u_{k-1}^K)
+B^{I,\gamma\alpha\beta}_{JK} \partial_\gamma (u_{k-1}^J-u_{k-2}^J)
\partial_\alpha\partial_\beta u_{k-1}^K.
\end{multline*}
With $h^{I,\alpha\beta}_K = - B^{I,\gamma\alpha\beta}_{JK}
\partial_\gamma u_{k-1}^J$, 
as above, \eqref{M_goal} implies \eqref{h_assumptions}.  We may, thus,
apply \eqref{high}.  Since $u_k-u_{k-1}$ has vanishing Cauchy data, it
suffices to bound
\begin{equation}\label{A_goal2} \|\la r\ra^{\frac{p+1}{2}}\Box_h
  Z^{\le N-1}(u_k-u_{k-1})\|_{L^2_tL^2_x} 
      \lesssim (M_{k-1}+M_{k-2})A_{k-1} + M_{k-1}A_k
      \end{equation}
as we may then apply \eqref{M_goal} and absorb $A_k$ to the other
side, which will yield \eqref{A_goal} as long as $\varepsilon$ is
sufficiently small.

      Using \eqref{A_null}, \eqref{B_null}, and \eqref{B_null2}, we have
      \begin{multline}\label{ABoxh}
         \|\la r\ra^{\frac{p+1}{2}}\Box_h
         Z^{\le N-1}(u_k-u_{k-1})\|_{L^2_tL^2_x} \\\lesssim
         \Bigl\|\la
  r\ra^{\frac{p+1}{2}} |Z^{\le N-1} \good
  (u_{k-1}-u_{k-2})| \Bigl(|Z^{\le \frac{N+1}{2}}
\partial  u_{k-1}|+| Z^{\le \frac{N-1}{2}}\partial
  u_{k-2}|\Bigr)\Bigr\|_{L^2_tL^2_x}
  \\
  +\Bigl\|\la r\ra^{\frac{p+1}{2}} |Z^{\le \frac{N-1}{2}}\good
  (u_{k-1}-u_{k-2})| \Bigl(|Z^{\le N} \partial  u_{k-1}|+|
  Z^{\le N-1}   \partial u_{k-2}|\Bigr)\Bigr\|_{L^2_tL^2_x} 
\\  +\Bigl\|\la r\ra^{\frac{p+1}{2}} |Z^{\le N-1} \partial (u_{k-1}-u_{k-2})|
  \Bigl(|Z^{\le \frac{N+1}{2}} \good
  u_{k-1}|+|Z^{\le \frac{N-1}{2}}\good
  u_{k-2}|\Bigr)\Bigr\|_{L^2_tL^2_x} 
  \\
  + \Bigl\|\la r\ra^{\frac{p+1}{2}} |Z^{\le \frac{N-1}{2}}\partial
  (u_{k-1}-u_{k-2})| \Bigl(|Z^{\le N}\good
  u_{k-1}|+|Z^{\le N-1}\good
  u_{k-2}|\Bigr)\Bigr\|_{L^2_tL^2_x}
  \\
  + \Bigl\|\la r\ra^{\frac{p+1}{2}} |Z^{\le \frac{N+1}{2}}\partial
  (u_{k}-u_{k-1})| | Z^{\le N-1} \good
  u_{k-1}|\Bigr\|_{L^2_tL^2_x}
\\  +\Bigl\|\la r\ra^{\frac{p+1}{2}} |Z^{\le N-1}\partial (u_{k}-u_{k-1})|
  |Z^{\le \frac{N-1}{2}}\good
  u_{k-1}|\Bigr\|_{L^2_tL^2_x} 
  \\
  +\Bigl\|\la r\ra^{\frac{p+1}{2}} |Z^{\le \frac{N+1}{2}}\good
  (u_{k}-u_{k-1})| |Z^{\le N-1} \partial  u_{k-1}|\Bigr\|_{L^2_tL^2_x} 
\\+\Bigl\|\la
  r\ra^{\frac{p+1}{2}} |Z^{\le N-1}\good
  (u_{k}-u_{k-1})| |Z^{\le \frac{N-1}{2}}\partial
  u_{k-1}|\Bigr\|_{L^2_tL^2_x}
\\
  + \Bigl\|\la r\ra^{\frac{p-1}{2}} |Z^{\le N-1}\partial
  (u_{k-1}-u_{k-2})| \Bigl(|Z^{\le \frac{N+1}{2}}\partial
  u_{k-1}|+|Z^{\le \frac{N-1}{2}}\partial u_{k-2}|\Bigr)\Bigr\|_{L^2_tL^2_x}
 \\
  + \Bigl\|\la r\ra^{\frac{p-1}{2}} | Z^{\le \frac{N-1}{2}}\partial
  (u_{k-1}-u_{k-2})| \Bigl(|Z^{\le N}\partial
  u_{k-1}|+ |Z^{\le N-1} \partial u_{k-2}|\Bigr)\Bigr\|_{L^2_tL^2_x}
\\+\Bigl\|\la
  r\ra^{\frac{p-1}{2}} |Z^{\le N-1}\partial
  (u_{k}-u_{k-1})| |Z^{\le \frac{N-1}{2}}\partial
  u_{k-1}|\Bigr\|_{L^2_tL^2_x}
 \\
  + \Bigl\|\la r\ra^{\frac{p-1}{2}} |Z^{\le \frac{N+1}{2}}\partial
  (u_{k}-u_{k-1})| |Z^{\le N-1}\partial
  u_{k-1}|\Bigr\|_{L^2_tL^2_x}.
    \end{multline}

We proceed with an argument that is akin to that used in the proof of
\eqref{M_goal}.  For each term, we apply \eqref{weighted_Sobolev} to
the lower order factor.  We then measure the ``good'' derivative
factor in a weighted $L^2_tL^2_x$ space, while the other factor is
placed into an energy space.  This approach, which is based in
\eqref{block1} and \eqref{block2}, shows that the first four
terms in the right side of \eqref{ABoxh} are bounded by
\begin{multline*}
  \|\la
  r\ra^{\frac{p-1}{2}} Z^{\le N-1}\good
  (u_{k-1}-u_{k-2})\|_{L^2_tL^2_x} \Bigl(\|Z^{\le \frac{N+5}{2}}\partial 
  u_{k-1}\|_{L^\infty_tL^2_x}+\|Z^{\le \frac{N+3}{2}}\partial 
  u_{k-2}\|_{L^\infty_tL^2_x}\Bigr)
  \\
  +\|\la r\ra^{\frac{p-1}{2}} |Z^{\le \frac{N+3}{2}}\good
  (u_{k-1}-u_{k-2})|\|_{L^2_tL^2_x} \Bigl(\|Z^{\le N}\partial  u_{k-1}\|_{L^\infty_tL^2_x}+|
  Z^{\le N-1} \partial u_{k-2}\|_{L^\infty_tL^2_x}\Bigr) 
\\  +\|Z^{\le N-1}\partial  (u_{k-1}-u_{k-2})\|_{L^\infty_tL^2_x}
  \Bigl(\|\la r\ra^{\frac{p-1}{2}} Z^{\le \frac{N+5}{2}}\good
  u_{k-1}\|_{L^2_tL^2_x}+\|\la r\ra^{\frac{p-1}{2}} Z^{\le \frac{N+3}{2}}\good
  u_{k-2}\|_{L^2_tL^2_x} \Bigr)
  \\
  + \| Z^{\le \frac{N+3}{2}}\partial 
  (u_{k-1}-u_{k-2})\|_{L^\infty_tL^2_x} \Bigl(\|\la
  r\ra^{\frac{p-1}{2}} Z^{\le N}\good
  u_{k-1}\|_{L^2_tL^2_x}+\|\la r\ra^{\frac{p-1}{2}} 
  Z^{\le N-1} \good
  u_{k-2}\|_{L^2_tL^2_x}\Bigr),
\end{multline*}
which, provided that $\frac{N+5}{2}\le N$, is
\[\lesssim A_{k-1}\Bigl(M_{k-1}+M_{k-2}\Bigr).\]
Similarly, the fifth through eighth terms in the right side of
\eqref{ABoxh} can be controlled by
\begin{multline*}
   \|Z^{\le \frac{N+5}{2}}\partial 
  (u_{k}-u_{k-1})\|_{L^\infty_tL^2_x} \|\la r\ra^{\frac{p-1}{2}}
  Z^{\le N-1}\good
  u_{k-1}\|_{L^2_tL^2_x}
\\  +\| Z^{\le N-1} \partial  (u_{k}-u_{k-1})\|_{L^\infty_tL^2_x}
  \|\la r\ra^{\frac{p-1}{2}} Z^{\le \frac{N+3}{2}}
  u_{k-1}\|_{L^2_tL^2_x} 
  \\
  +\|\la r\ra^{\frac{p-1}{2}} | Z^{\le \frac{N+5}{2}} \good
  (u_{k}-u_{k-1})\|_{L^2_tL^2_x} \| Z^{\le N-1} \partial u_{k-1}|\|_{L^\infty_tL^2_x} 
\\+\|\la
  r\ra^{\frac{p-1}{2}} |Z^{\le N-1} \good
  (u_{k}-u_{k-1})\|_{L^2_tL^2_x} \|Z^{\le \frac{N-1}{2}}\partial 
  u_{k-1}\|_{L^\infty_tL^2_x},
\end{multline*}
which in turn is $\lesssim A_k\cdot M_{k-1}$ provided that
$\frac{N+5}{2}\le N-1$.  Relying on the fact that $p<2$, the remaining terms in \eqref{ABoxh} (namely
the last four in the right side) are 
\begin{multline*}
\lesssim
  \|Z^{\le \frac{N+3}{2}}\partial 
  (u_{k-1}-u_{k-2})\|_{L^\infty_tL^2_x} \Bigl(\|Z^{\le N}
  u_{k-1}\|_{LE^1} + \|Z^{\le N-1} u_{k-2}\|_{LE^1}\Bigr)
\\+  \|Z^{\le N-1}\partial 
  (u_{k-1}-u_{k-2})\|_{L^\infty_tL^2_x} \Bigl(\|Z^{\le \frac{N+5}{2}}
  u_{k-1}\|_{LE^1} + \|Z^{\le \frac{N+3}{2}} u_{k-2}\|_{LE^1}\Bigr)
\\+\|Z^{\le N-1}\partial 
  (u_{k}-u_{k-1})\|_{L^\infty_tL^2_x} \|Z^{\le \frac{N+3}{2}}
  u_{k-1}\|_{LE^1}
\\+\|Z^{\le \frac{N+5}{2}}\partial 
  (u_{k}-u_{k-1})\|_{L^\infty_tL^2_x} \|Z^{\le N-1}
  u_{k-1}\|_{LE^1}
\end{multline*}
As this is 
\[\lesssim \Bigl(M_{k-2} + M_{k-1}\Bigr)A_{k-1}+M_{k-1}\cdot A_k,\]
we have completed the proof of \eqref{A_goal2}, which also completes
the proof of Theorem~\ref{thm:main}.




\begin{thebibliography}{10}

\bibitem{christodoulou_null}
Demetrios Christodoulou.
\newblock Global solutions of nonlinear hyperbolic equations for small initial
  data.
\newblock {\em Comm. Pure Appl. Math.}, 39(2):267--282, 1986.

\bibitem{DafRod}
Mihalis Dafermos and Igor Rodnianski.
\newblock A new physical-space approach to decay for the wave equation with
  applications to black hole spacetimes.
\newblock In {\em X{VI}th {I}nternational {C}ongress on {M}athematical
  {P}hysics}, pages 421--432. World Sci. Publ., Hackensack, NJ, 2010.

\bibitem{FMM}
Michael Facci, Alex Mcentarrfer, and Jason Metcalfe.
\newblock A $r^p$-weighted local energy approach to global existence for null
  form semilinear wave equations.
\newblock preprint, 2021.

\bibitem{katayama_kubo}
Soichiro Katayama and Hideo Kubo.
\newblock An alternative proof of global existence for nonlinear wave equations
  in an exterior domain.
\newblock {\em J. Math. Soc. Japan}, 60(4):1135--1170, 2008.

\bibitem{KSSsem}
Markus Keel, Hart~F. Smith, and Christopher~D. Sogge.
\newblock Almost global existence for some semilinear wave equations.
\newblock {\em J. Anal. Math.}, 87:265--279, 2002.
\newblock Dedicated to the memory of Thomas H. Wolff.

\bibitem{KSSquasobst}
Markus Keel, Hart~F. Smith, and Christopher~D. Sogge.
\newblock Global existence for a quasilinear wave equation outside of
  star-shaped domains.
\newblock {\em J. Funct. Anal.}, 189(1):155--226, 2002.

\bibitem{keir}
Joseph Keir.
\newblock The weak null condition and global existence using the p-weighted
  energy method.
\newblock {\em arXiv preprint arXiv:1808.09982}, 2018.

\bibitem{klainerman_null}
S.~Klainerman.
\newblock The null condition and global existence to nonlinear wave equations.
\newblock In {\em Nonlinear systems of partial differential equations in
  applied mathematics, {P}art 1 ({S}anta {F}e, {N}.{M}., 1984)}, volume~23 of
  {\em Lectures in Appl. Math.}, pages 293--326. Amer. Math. Soc., Providence,
  RI, 1986.

\bibitem{Klain}
Sergiu Klainerman.
\newblock Uniform decay estimates and the {L}orentz invariance of the classical
  wave equation.
\newblock {\em Comm. Pure Appl. Math.}, 38(3):321--332, 1985.

\bibitem{lindblad_nakamura_sogge}
Hans Lindblad, Makoto Nakamura, and Christopher~D. Sogge.
\newblock Remarks on global solutions for nonlinear wave equations under the
  standard null conditions.
\newblock {\em J. Differential Equations}, 254(3):1396--1436, 2013.

\bibitem{MNS_null}
Jason Metcalfe, Makoto Nakamura, and Christopher~D. Sogge.
\newblock Global existence of quasilinear, nonrelativistic wave equations
  satisfying the null condition.
\newblock {\em Japan. J. Math. (N.S.)}, 31(2):391--472, 2005.

\bibitem{MNS_multi_speed}
Jason Metcalfe, Makoto Nakamura, and Christopher~D. Sogge.
\newblock Global existence of solutions to multiple speed systems of
  quasilinear wave equations in exterior domains.
\newblock {\em Forum Math.}, 17(1):133--168, 2005.

\bibitem{MS_hyp_trap}
Jason Metcalfe and Christopher~D. Sogge.
\newblock Hyperbolic trapped rays and global existence of quasilinear wave
  equations.
\newblock {\em Invent. Math.}, 159(1):75--117, 2005.

\bibitem{MS_SIAM}
Jason Metcalfe and Christopher~D. Sogge.
\newblock Long-time existence of quasilinear wave equations exterior to
  star-shaped obstacles via energy methods.
\newblock {\em SIAM J. Math. Anal.}, 38(1):188--209, 2006.

\bibitem{MS_null_ext}
Jason Metcalfe and Christopher~D. Sogge.
\newblock Global existence of null-form wave equations in exterior domains.
\newblock {\em Math. Z.}, 256(3):521--549, 2007.

\bibitem{SdTu}
Thomas~C. Sideris and Shu-Yi Tu.
\newblock Global existence for systems of nonlinear wave equations in 3{D} with
  multiple speeds.
\newblock {\em SIAM J. Math. Anal.}, 33(2):477--488, 2001.

\bibitem{Sterb}
Jacob Sterbenz.
\newblock Angular regularity and {S}trichartz estimates for the wave equation.
\newblock {\em Int. Math. Res. Not.}, (4):187--231, 2005.
\newblock With an appendix by Igor Rodnianski.

\bibitem{Yang_quas}
Shiwu Yang.
\newblock On the quasilinear wave equations in time dependent inhomogeneous
  media.
\newblock {\em J. Hyperbolic Differ. Equ.}, 13(2):273--330, 2016.

\end{thebibliography}
\end{document}